\documentclass[a4paper,12pt, reqno]{amsart}
\usepackage{amssymb, amsmath}

\newtheorem{thm}{Theorem}[section]

\newtheorem{lem}[thm]{Lemma}
\newtheorem{prop}[thm]{Proposition}
\theoremstyle{definition}
\newtheorem{defn}[thm]{Definition}
\theoremstyle{remark}
\newtheorem{rem}[thm]{\bf Remark}

\numberwithin{equation}{section}
\newcommand{\ds}{\displaystyle}

\newcommand\ea{\begin{equation}}
\newcommand\ez{\end{equation}}
\newcommand\eas{\begin{equation*}}
\newcommand\ezs{\end{equation*}}
\def\split{\begin{array}{ll}}
\def\endsplit{\end{array}}

\def\m{(m)}

\newcommand{\C}{\mathbb C}

\newcommand{\cR}{\mathcal R}

\def\ol{\overline}
\def\pa{\partial}
\newcommand{\al}{\alpha}

\def\gr{\nabla}

\def\pa{\partial}
\newcommand{\q}{\hskip 20pt}
\def\s{\sigma}

\newcommand{\wpsi}{\widetilde{\psi}}

\newcommand{\BP}{\vskip 10pt}
\newcommand{\SP}{\vskip 5pt}

\newcommand{\wa}{\widetilde{a}}
\newcommand{\wb}{\widetilde{b}}
\newcommand{\wc}{\widetilde{c}}

\begin{document}

\title[Dirac operators]{A sequence of zero modes of Weyl-Dirac operators 
and an associated sequence of solvable polynomials
}

\thanks{${}^{\ddag}$ {\scriptsize Supported by 
     Grant-in-Aid for Scientific Research (C) 
    No.  21540193, 
   Japan Society for the Promotion of Science.}}

\maketitle
  \centerline{\it Dedicated to Professor Edmunds and Professor Evans}
  \centerline{\it on the occasion of their eightieth and seventieth birthdays}
  
\vspace{30pt}

\centerline{\author{Yoshimi Sait\={o}${}^{\dag}$ and 
Tomio Umeda${}^{\ddag}$}}

\vspace{6pt}

\small{
\centerline{${}^{\dag}$\it Department of Mathematics, University of Alabama at Birmingham} 
 
\centerline{\it Birmingham, AL 35294, USA}
 
\vspace{6pt}

\centerline{${}^{\ddag}$\it Department of Mathematical Sciences,
University of Hyogo}
\centerline{\it Himeji 671-2201, Japan}

\vspace{6pt}
\centerline{${}^{\dag}$\it saito@math.uab.edu}

\vspace{4pt}
\centerline{${}^{\ddag}$\it umeda@sci.u-hyogo.ac.jp}
}

\vspace{12pt}

\begin{quote}
{\small {\bf Abstract.} It is shown that
a series of solvable polynomials is attached to
the series of zero modes constructed by Adam, Muratori and Nash 
\cite{AdamMuratoriNash1}.

\vspace{8pt}
\noindent
{\it Keywords}: Weyl-Dirac operators, magnetic potentials, zero modes,
solvable polynomials.

\vspace{8pt}
\noindent
Mathematics Subject Classification 2000: 35Q40, 35P99, 11R09

}
\end{quote}

\vskip 30pt

\section{\bf Introduction}

The aim of this note is to point out
 an interesting and unpredictable connection
between zero modes and solvable polynomials. 
We shall precisely explain our aim. 

To this end, we first introduce a Weyl-Dirac
operator
\begin{equation}\label{1weylD}     
  H_A = \s\cdot(D - A) = \sum_{k=1}^3 \s_k(D_k - A_k(x)),
\end{equation}
where $\s= (\s_1, \, \s_2, \, \s_3)$ is
the triple of  $2 \times 2$ Pauli matrices
\begin{equation*}\label{eqn:1-3}
\sigma_1 =
\begin{pmatrix}
0&1 \\ 1& 0
\end{pmatrix}, \,\,\,
\sigma_2 =
\begin{pmatrix}
0& -i  \\ i&0
\end{pmatrix}, \,\,\,
\sigma_3 =
\begin{pmatrix}
1&0 \\ 0&-1
\end{pmatrix},
\end{equation*}
$A(x) = (A_1(x), A_2(x), A_3(x))$ is a  vector
potential, and 
\begin{equation*}
 D =  -i\gr
   = 
 \Big(-i\frac{\pa}{\pa x_1}\,,\, -i\frac{\pa}{\pa x_2}\,,\, 
  -i\frac{\pa}{\pa x_3}\Big).
\end{equation*}
If each component of the vector potential $A$ is a bounded measureable function,
then the operator $\sigma \cdot A$ is a bounded self-adjoint operator in
the Hilbert space $\mathcal L^2 := [L^2({\mathbb R}^3)]^2$. 
Hence it is straightforward that
the Weyl-Dirac operator $H_A$ defines
the unique self-adjoint realization in  $\mathcal L^2$
and its domain is given as 
  $\mbox{Dom}(H_A)=\mathcal H^1:= [H^1({\mathbb R}^3)]^2$ 
whenever $A_j \in L^{\infty}({\mathbb R}^3)$. Here $H^1({\mathbb R}^3)$ denotes
the Sobolev space of order 1.

\begin{defn}
If $\psi \in \mbox{Ker}(H_A)$, then $\psi$ is called 
a zero mode of $H_A$. In other words,
$\psi$ is said to be  a zero mode
if and only if $\psi \in \mbox{Dom}(H_A)$ and $H_A \psi = 0$.
\end{defn}

We should remark that the Weyl-Dirac operator
is intimately related with the Pauli operator
\begin{equation*}
P_A= \sum_{j=1}^3 ( D_j - A_j)^2 - \sigma \cdot B,
\end{equation*}
 where $B$ denotes the magnetic field given
 by $B=\nabla \times A$. 
 This is because 
 \begin{equation*}
 P_A = \{ \sigma \cdot (D - A) \}^2 = H_{\!A}^2
 \end{equation*}
 in a formal sense.
Roughly speaking, we can say that $\psi$ is a zero mode
of the Weyl-Dirac operator $H_A$ 
if and only if it is a zero mode of the Pauli operator
$P_A$. 

It is now well understood that
the existence of magnetic fields which give rise to zero
modes of the Weyl-Dirac operators has significant 
implications in mathematics and physics
(see \cite{AdamMuratoriNash1}, \cite{AdamMuratoriNash2}, 
\cite{AdamMuratoriNash3}, \cite{BalinEvan1},
\cite{BalinEvan2}, \cite{BalinEvan3}, \cite{BalinskyEvansUmeda},
\cite{Elton-1}, \cite{Elton},
\cite{ErdosSolovej},
\cite{FrohlichLiebLoss}, 
\cite{LossYau},  \cite{Pickl}, \cite{PicklDurr},
\cite{PicklDurr2}, \cite{SaitoUmeda1}).
However, Balinsky and Evans \cite{BalinEvan1, BalinEvan2, BalinEvan3}
and Elton \cite{Elton}
showed that
the set of vector potentials which yield zero modes
is scarce in a certain sense.

Vector potentials which give rise to zero modes do exist.
The first examples of such vector potentials were
given by Loss and Yau \cite{LossYau}. Later Adam, Muratori and Nash 
\cite{AdamMuratoriNash1, AdamMuratoriNash2, AdamMuratoriNash3} and
Elton \cite{Elton-1} constructed further examples of
zero modes, using and developing the ideas from \cite{LossYau}.
The works \cite{ErdosSolovej} by Erd\"os and Solovej generalize
all these examples.

The basic idea of Loss and Yau \cite{LossYau}
is to find a solution of the Loss-Yau equation 
\begin{equation}   \label{eqn:LossYauEqn}
          (\s\cdot D)\psi(x) = h(x)\psi(x),
\end{equation}
where $h$ is a given (real-valued) function, and then
to define a vector potential $A$ so that $\psi$ satisfies the 
equation
$\sigma \cdot (D -A)\psi=0$.
Precise statement of their idea is the following.

\begin{prop}   \label{prop:LY}                     
        {\rm (Loss-Yau \cite{LossYau})}. 
Let $\psi \in {\mathcal H}^1$ be
a solution to the Loss-Yau equation {\rm(\ref{eqn:LossYauEqn})}
with a real valued bounded function $h$.
Then $\psi$ is a zero mode of 
the Weyl-Dirac operator $H_A$ 
with the vector potential defined  by
\begin{equation}  \label{eqn:LossYau1}
  A(x) 
= \frac{h(x)}{|\psi(x)|^2}
\big( \psi(x)\cdot\s_1\psi(x), \, \psi\cdot\s_2\psi(x),\,
                                         \psi\cdot\s_3\psi(x) \big),
\end{equation}
 where, for $a = {}^t(a_1, a_2), b = {}^t(b_1, b_2) \in \C^2$, $a\cdot b$ denotes the
inner product:
\begin{equation*}
            a \cdot b = \ol{a}_1\,b_1 + \ol{a}_2\,b_2.
\end{equation*}
\end{prop}

In \cite{LossYau}, choosing 
\begin{equation}   \label{eqn:LossYau2}
\psi(x)= \langle x \rangle^{-3} (I_2 + i\sigma \cdot x) \phi_0
\qquad (\langle x \rangle=\sqrt{1+|x|^2} \,),
\end{equation}
where $I_2$ is a $2\times 2$ unit matrix and
 $\phi_0 \in \C^2$ a unit
vector, they showed that $\psi$ defined by (\ref{eqn:LossYau2})
satisfies the Loss-Yau equation (\ref{eqn:LossYauEqn}) with
\begin{equation}  \label{eqn:LossYau3}
h(x)= \frac{3}{\langle x \rangle^2}.  
\end{equation}
It follows from (\ref{eqn:LossYau1}) and (\ref{eqn:LossYau2})
that
\begin{align}      \label{eqn:LY-1}
A(x) = 3 \langle x \rangle^{-4} 
  \big\{ (1-|x|^2)  w_0 + 2 (w_0 \cdot x) x + 2 w_0 \times x \big\},
\end{align}
where
\begin{equation}      \label{eqn:LY-1-1}
w_0=  \phi_0 \cdot (\sigma \phi_0 )
:=\big( 
\phi_0 \cdot(\sigma_1 \phi_0), \, 
  \phi_0 \cdot (\sigma_2 \phi_0), \,
\phi_0 \cdot\sigma_3 \phi_0)
 \big),
\end{equation}
and $w_0 \cdot x$ and $w_0 \times x$ denote
the inner product and the exterior product respectively.
Proposition \ref{prop:LY}  implies that
$\psi$ defined by (\ref{eqn:LossYau2}) is a zero mode
of the Weyl-Dirac operator $H_A$ 
with the vector potential (\ref{eqn:LY-1}).

Adam, Muratori and Nash \cite{AdamMuratoriNash1} exploited 
the idea of Proposition \ref{prop:LY}, and successfully 
constructed a series  $\big\{ \psi^{(m)} \big\}_{m=1}^{\infty}$,
each of which satisfies the Loss-Yau equation (\ref{eqn:LossYauEqn})
with
\begin{equation}  \label{eqn:LossYau4}
h^{(m)}(x)= \frac{2m +3}{\langle x \rangle^2} \qquad (m=1, \, 2, \, \cdots).  
\end{equation}
It is obvious that each $ \psi^{(m)}$ is a zero mode
of the Weyl-Dirac operator $H_{A^{(m)}}$
with the vector potential
\begin{equation*}  \label{eqn:LossYau5}
  A^{(m)}(x) 
= \frac{h^{(m)}(x)}{|\psi^{(m)}(x)|^2}
\big( \psi^{(m)}(x)\cdot\s_1\psi^{(m)}(x), \, \psi^{(m)}\cdot\s_2\psi^{(m)}(x),\,
                  \psi^{(m)}\cdot\s_3\psi^{(m)}(x) \big).
\end{equation*}

The goal of this note is to show that 
a polynomial $P_m(t)$ of degree $m+1$ is associated with
each zero mode $\psi^{(m)}$ in such a way that
the polynomial equation $P_m(t)=0$ is solvable and
 all of the roots of this equation  
determine a set of zero modes, one of which is designated as  $\psi^{(m)}$. 
Obviously, as $m$ gets larger, it will become more difficult
to solve the equation $P_m(t)=0$. 
It is well-known \cite{Galois} that \lq\lq{\it there is 
 no formula for the roots of a fifth (or higher) degree polynomial equation in terms of
the coefficients of the polynomial, using only the usual algebraic operations (addition,
subtraction, multiplication, division) and application of radicals 
{\rm(}square roots, cube roots, etc{\rm)}.}\rq\rq
\ Here are the first six equations of $P_m(t)=0$:
\begin{align}
{}& P_1(t) = 0 \Longleftrightarrow 9t^2 - 34t + 25 = 0,  \nonumber\\
\noalign{\vskip 2pt}
{}& P_2(t) = 0 \Longleftrightarrow 81t^3 - 747t^2 + 1891t - 1225 = 0, 
\nonumber\\
\noalign{\vskip 2pt}
{}& P_3(t) = 0  \Longleftrightarrow 81t^4 -1476t^3+8614t^2-18244t+11025= 0,  \nonumber\\
\noalign{\vskip 2pt}
{}& P_4(t)=0  \Longleftrightarrow  729t^5 -23085t^4 +256122t^3 -1206490t^2  \nonumber\\
{}& \qquad \qquad\qquad \qquad  \qquad\qquad \qquad  \qquad        
+2306749t -1334025=0,  \nonumber    \\
\noalign{\vskip 2pt}
{}& P_5(t)=0  \Longleftrightarrow
  6561t^6 -330966t^5 + 6206463t^4  -54143028t^3 \nonumber \\
{}& \qquad \qquad\qquad \qquad  \qquad
 + 224657551t^2 - 401846806t + 225450225  =0,  \nonumber  \\
\noalign{\vskip 2pt}
{}& P_6(t)=0  \Longleftrightarrow
 6561t^7 -494991t^6 +14480613t^5 -209304603t^4 +1578233251t^3 \nonumber \\
{}& \qquad \qquad\qquad \qquad  \qquad \qquad
 -6018285581t^2+10271620375t -5636255625=0.  \nonumber
\end{align}

It is incredable to see that all these polynomial equations are solvable.
Actually, by computer-aided calculation we see that
\begin{align}
{}& P_1(t) = 0 \Longleftrightarrow t=1, \, \frac{25}{9}   \nonumber\\
\noalign{\vskip 4pt}
{}& P_2(t) = 0 \Longleftrightarrow  
t=1, \,   \frac{25}{9} , \, \frac{49}{9} , 
\nonumber\\
\noalign{\vskip 4pt}
{}& P_3(t) = 0  \Longleftrightarrow 
t=1, \,   \frac{25}{9} , \,  \frac{49}{9} , 
 \, 9,  \nonumber\\
\noalign{\vskip 4pt}
{}& P_4(t)=0  \Longleftrightarrow  
t=1, \,  \frac{25}{9} , \, \frac{49}{9} , 
 \, 9, \,   \frac{121}{9} ,  \nonumber\\
\noalign{\vskip 4pt}
{}& P_5(t)=0  \Longleftrightarrow
t=1, \,   \frac{25}{9} , \, \frac{49}{9} , 
 \, 9, \,    \frac{121}{9} ,  
\,    \frac{169}{9} ,\nonumber  \\
\noalign{\vskip 4pt}
{}& P_6(t)=0  \Longleftrightarrow
t=1, \,    \frac{25}{9} , \, \frac{49}{9} , 
 \, 9, \,   \frac{121}{9} ,  
\,   \frac{169}{9} , \, 25.  \nonumber
\end{align}
Based on this observation, it is natural
to predict that 
the roots of the equation $P_m(t)=0$ must be given by
\begin{equation}  \label{eqn:roots-1} 
1, \, \Big( \frac{5}{3} \Big)^2, \, \Big( \frac{7}{3} \Big)^2, \,
\cdots, \, \Big( \frac{2m+3}{3} \Big)^2
\end{equation}
for every $m \in \mathbb N$.
This prediction will be proven to be true in \S \ref{sec:Conjecture},
though we should like to mention that these 
polynomials $P_m(t)$ will be only implicitly defined  
in a rather messy manner; see the formula ($L_m$) in
Proposition \ref{prop:Prop2.1}  as well as   Proposition \ref{prop:Prop2.2} 
in \S \ref{sec:Recurrence}. 
In relation with this, we emphasize that
 the bigger $m$ gets, the more complicated
$P_m(t)$ becomes, as can be seen from  $P_1(t)$, $\cdots$, $P_6(t)$ above.

We should like to call  $P_m(t)$ in their monic forms
the {\it Adam-Muratori-Nash polynomials}. 
We feel solvability of $P_m(t)$ seems an interesting subject
from the view point of Galois theory (see Edwards \cite{Edwards}),
though it is well beyond the scope of the present note.

\vspace{10pt}

%%%%%%%%%%%%%%%%%%%%%%%%%%%%%%%%%%%%%%%%%%%%%%%%%%%%%%%%%
%%%%%%%%%%%%%%%%%%%%%%%%%%%%%%%%%%%%%%%%%%%%%%%%%%%%%%%%%
%%%%%%%%%%%%%%%%%%%%%%%%%%%%%%%%%%%%%%%%%%%%%%%%%%%%%%%%%
\section{\bf Recurrence formulae}  
\label{sec:Recurrence}

 In this section we follow the line of the arguments
demonstrated in Adam-Muratori-Nash \cite{AdamMuratoriNash1}.
For this reason, we shall use the same notation as in 
 \cite{AdamMuratoriNash1} 
to indicate
$I_2$ and $i\sigma \cdot x$ in the rest of this note ; namely
\begin{equation} \label{eqn:amn1}
{\mathbf 1}= I_2, \ \ {\mathbf X} = i\sigma \cdot x .
\end{equation}
Their construction of the zero modes  is based on the following
ansatz:
\begin{equation}\label{1psim}                                                        
    \psi^{(m)}(x) 
= \langle     x \rangle  ^{-(3+2m)}\Big[\Big(\sum_{n=0}^{m} a_n|x|^{2n}\Big){\mathbf 1} 
                        + \Big(\sum_{n=0}^{m} b_n|x|^{2n}\Big){\mathbf X}\Big]\phi_0,
\end{equation}
   where $\langle     x \rangle   = \sqrt{1 + |x|^2}$, 
 and $\phi_0 = {}^t(1, 0)$. We are going to study the
     case where  $h(x)$ in (\ref{eqn:LossYauEqn})  is  $\langle     x \rangle  ^{-2}$ 
multiplied by  a constant $\al$. By a simple but tedious computation we
have

\begin{prop}                                                         
\label{prop:Prop2.1}
        Let $\psi^{\m}(x)$ be as above with $a_0 = 1$.

\SP

        {\rm (i)} Then, for each $m = 1,\, 2, \, 3, \,\cdots,$ we have
\begin{multline}\label{1hapsi}                                               \qquad    (\s\cdot D)\psi^{\m} \\
      = \langle     x \rangle  ^{-(5+2m)}\sum_{n=0}^m \Big[ (2m + 3)b_n|x|^{2n} \hskip 80pt \\
      \ds \hskip 130pt                  - (2m - 2n)b_n|x|^{2(n+1)}\Big]{\mathbf 1}\phi_0 \\
      \ds \hskip 50pt + \langle     x \rangle  ^{-(5+2m)}\Big[\sum_{n=0}^{m-1} \{(3+2m)a_n
                                                          - 2(n+1)a_{n+1}\}|x|^{2n}  \\
      \ds \hskip 20pt + (3 + 2m)a_m|x|^{2m} - \sum_{n=1}^{m-1} 2(n+1)a_{n+1}|x|^{2(n+1)}\Big]
   {\mathbf X}\phi_0.
\end{multline}

\SP
        {\rm (ii)} The equation
\begin{equation}\label{1WDeq}                                                           (\s\cdot D)\psi^{\m}(x) = \frac{\alpha}{\langle     x \rangle  ^2}\psi^{\m}(x)
\end{equation}
     where $\alpha$ is a constant, is equivalent to the system
\begin{equation*}
(L_m) \quad 
\begin{cases}
       (2j - 1): \ \ \ \ \, 2ja_{j} - (2m + 5 - 2j)a_{j-1} = - 3b_0b_{j-1}  \\
       \hskip 172pt                                    (1 \le j \le m, a_0 = 1), \\
       (2k): \ \ \ \ \ \ \ \ \ (2k + 3)b_k - (2m + 2 - 2k)b_{k-1} = 3b_0a_k \\
       \hskip 230pt                                                 (1 \le k \le m), \\
       (2m + 1): \ \ \ \,  a_{m} = b_0b_{m}
\end{cases}
\qquad\qquad\qquad           
\end{equation*}
     of $(2m + 1)$ equations for the $(2m + 1)$ unknowns $a_j \ (1 \le j \le m)$
     and $b_k \ (0 \le k \le m)$, where the constant $\alpha$ turns out to be $3b_0$.
\end{prop}                                                            

        It is easy to see that $\{ a_n \}_{n=1}^m$ and $\{ b_n \}_{n=1}^m$ can be expressed
     by $b_0$ and the last equation $a_m = b_0b_m$ becomes 
     a polynomial equation for unknown $b_0$.

\begin{prop}                                                            
\label{prop:Prop2.2}  
        The last equation $(2m + 1)$ of  
        {\rm(}$L_m${\rm)\, : \;}$a_{m} = b_0b_{m}$ takes 
        the form $P_m(b_0^2) = 0$,
     where $P_m(t)$ is a polynomial of degree $m+1$.
\end{prop}   
                                                     
%%%%%

\begin{proof}  
Since $a_0=1$, the first equation (1) of  ($L_m$) is given as $2a_1 -(2m+3)=-3b_0^2$.
Hence   $a_1 = p_1(b_0^2)$, where   $p_1(t)= 2^{-1}\{ (2m+3) - 3t\}$.                                               
        From the second equation (2) of 
        ($L_m$) we see that $b_1$ takes the form
     $b_1 = b_0q_1(b_0^2)$, where $q_1(t)=(10)^{-1}(10m + 9 -9t)$. 
     Then, by induction, one can show that $a_j$, $1 \le j \le m$, and
    $b_k$, $1 \le k \le m$, are expressed as $a_j = p_j(b_0^2)$ and $b_k = b_0q_k(b_0^2)$ 
    with polynomial $p_j(t)$  of degree $j$ and 
    polynomial $q_k(t)$ of   degree $k$. Thus the equation
    $b_0b_m - a_m = 0$ becomes a polynomial equation $P_m(b_0^2) = 0$,
     where $P_m(t)$ is a polynomial of $t$
    of degree $m+1$.
\end{proof}                                                      

\vspace{4pt}

\begin{rem}                                                              
        The case $m = 1$ is discussed in
     \cite{AdamMuratoriNash1}.
 In this case 
\begin{align}    \label{eqn:AMN-3-SaUmRmk}  
\psi^{(1)}(x) =\langle x \rangle^{-5}   
\big\{  ( 1 -\frac{5}{3} |x|^2) {\mathbf 1}  
  + (\frac{5}{3} - |x|^2) \mathbf X \big\} \phi_0.   
\end{align}
\end{rem}                                                                 

\vspace{4pt}

\begin{prop}    
\label{prop:prop2.4}                                                      
        Let $m$ be a fixed nonnegative integer and 
let $a_n$ and $b_n$ be the coefficients in
     {\rm (\ref{1psim})} with $a_0=1$. Then we have
\begin{equation}\label{2amnbmn}                                                   
     \left(
     \begin{matrix}
           a_j \\
           \ \  \\
           b_j
     \end{matrix}
     \right)
          = K_{j}K_{j-1} \cdots K_2
     \left(
     \begin{matrix}
           a_1 \\
           \ \ \\
           b_1
     \end{matrix}
     \right) \q (j = 2, 3, \cdots, m),
\end{equation}
     where
\begin{equation}\label{2kmj}                                                        
       K_p =
     \left(
     \begin{matrix}
        \ds \frac{2m + 5 - 2p}{2p}               & \ds -\frac{3b_0}{2p}                \\
     \noalign{\vskip 4pt}
        \ds \frac{3(2m + 5 - 2p)b_0}{2p(2p + 3)} & \ds \frac{2p(2m + 2 -2p) - 9b_0^2}{2p(2p + 3)}
     \end{matrix}
     \right)
\end{equation}
     for $p = 2, 3, \cdots, m$, and
\begin{equation}\label{2am1bm1}                                                          
     \left(
     \begin{matrix}
        \ds a_1 \\
         \ \ \\
        \ds b_1
     \end{matrix}
     \right) =
     \left(
     \begin{matrix}
               \ds \frac{2m + 3 - 3b_0^2}2 \\
    \noalign{\vskip 4pt}
               \ds \frac{b_0}{10}(10m + 9 - 9b_0^2)
     \end{matrix}
     \right).
\end{equation}
\end{prop}                                                               

\vspace{4pt}

\begin{proof}   
We divide the proof into three steps.
                                                         
        (I) Let $2 \le j \le m$. It follows from the equation $(2k)$ in 
        the system  ($L_m$) with $k$ replaced by $j$
     that
\begin{equation} \label{2Lk}                                                              
         (2j+ 3)b_j - (2m + 2 -2j)b_{j-1} = 3b_0a_j
\end{equation}
From the equation $(2j-1)$ in ($L_m$) we have
\begin{equation*}
          2ja_j - (2m + 5 - 2j)a_{j-1} = - 3b_0b_{j-1},
\end{equation*}
     or
\begin{equation}\label{2Lj}                                                                
           a_j = \frac{(2m + 5 - 2j)a_{j-1} - 3b_0b_{j-1}}{2j}\,.
\end{equation}
Then we obtain from (\ref{2Lk}) and (\ref{2Lj})
\begin{multline*}
 \qquad  \quad   (2j + 3)b_j - (2m + 2 -2j)b_{j-1} \\
 \noalign{\vskip 4pt}
           = \frac{3b_0\big[(2m + 5 - 2j)a_{j-1} - 3b_0b_{j-1})\big]}{2j}\,,
                                                                                      \hskip 50pt
\end{multline*}
or
\begin{align} \label{2bj}   
    {}&   (2j + 3)b_j  \nonumber\\
 \noalign{\vskip 4pt}
      & \qquad = \frac{3b_0(2m + 5 - 2j)}{2j} a_{j-1} 
                + \frac{2j(2m + 2 - 2j) - 9b_0^2}{2j} b_{j-1}.  
\end{align}

\SP

  (II) From the equation $(2j-1)$ in ($L_m$) with $j = 1$
   we see that, by noting that $a_0 = 1$,
\begin{equation*}
        2a_1 - (2m + 3) = -3b_0^2,
\end{equation*}
and hence
\begin{equation}\label{2a1}                                                                 
          a_1 = \frac{2m + 3 - 3b_0^2}2\,.
\end{equation}
We have from the equation $(2k)$ in ($L_m$) with $k = 1$
\begin{equation*}
           5b_1 - (2m + 2 - 2)b_0 = 3b_0a_1,
\end{equation*}
which is combined with (\ref{2a1}) to yield
\begin{equation*}
           5b_1 - 2m b_0 = \frac{3b_0\big(2m + 3 - 3b_0^2\big)}2,
\end{equation*}
or
\begin{equation}\label{2b1}                                                                   
           b_1 = \frac{b_0\big(10m + 9 - 9b_0^2\big)}{10}\,.
\end{equation}
\SP

 (III) It follows from (\ref{2Lj}) and (\ref{2bj}) that
\begin{equation}\label{2ajbj}                                                                
     \left(
     \begin{matrix}
        \ds a_j \\
         \ \ \\
        \ds b_j
     \end{matrix}
     \right) = K_j
     \left(
     \begin{matrix}
        \ds a_{j-1} \\
         \ \ \\
        \ds a_{j-1}
     \end{matrix}
     \right) \q (j = 2, 3, \cdots, m)
\end{equation}
with $K_j$ given by (\ref{2kmj}) with $p$ replaced by $j$. By using (\ref{2ajbj}) repeatedly
     we can obtain (\ref{2amnbmn}). As for $a_1$ and $b_1$, (\ref{2am1bm1}) is justified by
     (\ref{2a1}) and (\ref{2b1}).
\end{proof}

\begin{rem}  
Proposition \ref{prop:prop2.4} was used to construct a Maple program to find the
the polynomial  $P_m(t)$ as well as 
to solve 
the polynomial equation $P_m(t)=0$.   
     We have been able to handle the equations with the Maple
     program up to the case 
$m =26$.  The first six equations were listed up at the end of \S 1.

\end{rem}

%%%%%%%%%%%%%%%%%%%%%%%%%%%%%%%%%%%%%%%%%%%%%%%%%%%%%%%%%%%%%%%%%%%%%%%%%%%%%%%%%%%
%%%%%%%%%%%%%%%%%%%%%%%%%%%%%%%%%%%%%%%%%%%%%%%%%%%%%%%%%%%%%%%%%%%%%%%%%%%%%%%%%%%
%%%%%%%%%%%%%%%%%%%%%%%%%%%%%%%%%%%%%%%%%%%%%%%%%%%%%%%%%%%%%%%%%%%%%%%%%%%%%%%%%%%

\section{\bf Monotonicity of the sequence 
${\mathbf\{{\cR_m}\}_{m=1}^{\infty}}$}
\label{sec:Monotonicity}

We begin by 

\SP

\begin{defn}
 For each $m \in \mathbb N$, $\cR_m$ is defined to
 be the set of all the roots of the polynomial
equation $P_m(t)=0$, namely,
\begin{equation} \label{eqn:roots-2-def}
\cR_m := \big\{ \, t \in \mathbb C \; \big| \, P_m(t)=0 \, \big\}.
\end{equation}
\end{defn}

\BP

\begin{prop}                                                        
\label{prop:Prop3.1}  
        Let $\cR_m$ be as above. Then we have
\begin{equation}\label{3increas}                                            
         \cR_1 \subset \cR_2 \subset \cdots \subset \cR_m \subset \cdots \,,
\end{equation}
     i.e., the sequence $\{ \cR_m \}_{m=1}^{\infty}$ is increasing with $m$.
\end{prop}                                                      

\begin{proof}                                                    
        For $m = 1, 2, 3, \cdots,$ let $\psi^{(m)}(x)$ be given by (\ref{1psim}) and suppose
     that the coefficients of $\psi^{(m)}(x)$ satisfy the system ($L_m$), {\it i.e.},
     $\psi^{(m)}(x)$ is a solution to  the  equation (\ref{1WDeq}) and 
     hence a zero mode of the Weyl-Dirac operator $H_{A^{(m)}}$. 
     Thus $b_0^2\in\cR_m$. 
     
  We now rewrite
     $\psi^{(m)}(x)$ as
\begin{eqnarray}\label{3psim1}                                                        
    \hskip 30pt \psi^{(m)}(x)&=&\langle x \rangle^{-(3+2(m+1))}
            \Big[
       \langle x \rangle^2 \Big(\sum_{n=0}^{m} a_n|x|^{2n}\Big) {\mathbf 1} 
       +   \langle x \rangle^2 \Big(\sum_{n=0}^{m} b_n|x|^{2n}\Big)
       {\mathbf X}\Big]
       \phi_0    \notag \\
       &=:& \wpsi^{(m+1)}(x). 
\end{eqnarray}
     By using the definition $\langle     x \rangle  ^2 = 1 + |x|^2$, we obtain
\begin{equation}\label{3psim2}                                                        
    \wpsi^{(m+1)}(x) = \langle x \rangle^{-(3+2(m+1))}
\Big[\Big(\sum_{n=0}^{m+1} \wa_n|x|^{2n}\Big){\mathbf 1} 
                   + \Big(\sum_{n=0}^{m+1} \wb_n|x|^{2n}\Big)
       {\mathbf X}\Big]\phi_0,
\end{equation}
     where
\begin{equation}\label{3coeff}                                                          
    \begin{cases}
         \wa_0 = a_0=1, \\
         \wa_n = a_{n-1} + a_{n}\q (1 \le n \le m), \\
         \wa_{m+1} = a_m, \\
         \wb_0 = b_0, \\
         \wb_n = b_{n-1} + b_n \q (1 \le n \le m),\\
         \wb_{m+1} = b_m.
    \end{cases}
\end{equation}
     Therefore, noting that $\psi^{(m)}(x)$ satisfies the equation
\begin{equation*}
 (\sigma\cdot D)\psi^{(m)}(x) =   3b_0 \langle x \rangle^{-2}\psi^{(m)}(x)
\end{equation*}
     and that $\wb_0 = b_0$
     by (\ref{3coeff}), we see that
\begin{eqnarray}\label{eq1}                                                          
    (\sigma\cdot D)\wpsi^{(m+1)}(x) &=& (\sigma\cdot D)\psi^{(m)}(x)  \notag\\
                                    & = &3b_0\langle     x \rangle  ^{-2}\psi^{(m)}(x) \\
              &=& 3\wb_0 \langle x \rangle^{-2}\wpsi^{(m+1)}(x).         \notag
\end{eqnarray}
     Thus, since $\wpsi^{(m+1)}(x)$ is a zero mode, we see that the coefficients $\wa_n$ and
     $\wb_n$ satisfy the system ($L_{m+1}$) which is ($L_m$) in Proposition \ref{prop:Prop2.1}(ii) with
     $m$ replaced  by $m + 1$. Thus we have the system of (2m + 3) equations for
\begin{equation*}
\begin{cases}
       (2j - 1) \ \ \ \  2j\wa_{j} - (2m + 7 - 2j)\wa_{j-1} = -
3\wb_0\wb_{j-1}  \\
       \hskip 192pt                                    (1 \le j \le m + 1, \ \wa_0 = 1), \\
       (2k) \ \ \ \ \ \ \ \ \ (2k + 3)\wb_k - (2m + 4 - 2k)\wb_{k-1} = 3\wb_0\wa_k \\
       \hskip 230pt                                                 (1 \le k \le m + 1), \\
       (2m + 3) \ \ \ \  \wa_{m+1} = \wb_0\wb_{m+1}.
\end{cases}
\end{equation*}
     Therefore $b_0 = \wb_0$ satisfies the polynomial equation $P_{m+1}(b_0^2) = 0$.
 \end{proof}

\begin{rem}  \label{rem:Rem3.3}                                                         
        Proposition \ref{prop:Prop3.1} above does not give
        us enough information to determine the set $\cR_m$ 
though it significantly clarifies
     the situation. In fact, we know that
\begin{equation}  \label{eqn:induction1}
\cR_m= \Big\{ \, 1, \, 
\Big( \frac{5}{3} \Big)^2, \, \Big( \frac{7}{3} \Big)^2, \,
\cdots, \, \Big( \frac{2m+3}{3} \Big)^2 \Big\}
\end{equation}
for $m=1$, $\cdots$, $6$. 
We also know that the polynomial $P_m(t)$ is of degree $m+1$.
Therefore Proposition \ref{prop:Prop3.1}, together with 
these two facts, tells us 
that we can prove (\ref{eqn:induction1}) by induction on $m$.
In fact,
assuming that (\ref{eqn:induction1}) with $m$ replaced by
$m-1$ is true,
 we only have to show that 
\begin{equation}   \label{eqn:induction2}
\Big( \frac{2m+3}{3}  \Big)^2 \in \cR_m
\end{equation}
for every $m \ge 2$ (actually $m\ge 7)$.
\end{rem}                                                             

\BP

%%%%%%%%%%%%%%%%%%%%%%%%%%%%%%%%%%%%%%%%%%%%%%%%%%%%%%%%%%%%%%%%%%%%%%%%%%%%%%%%%%%
%%%%%%%%%%%%%%%%%%%%%%%%%%%%%%%%%%%%%%%%%%%%%%%%%%%%%%%%%%%%%%%%%%%%%%%%%%%%%%%%%%%
%%%%%%%%%%%%%%%%%%%%%%%%%%%%%%%%%%%%%%%%%%%%%%%%%%%%%%%%%%%%%%%%%%%%%%%%%%%%%%%%%%%

\section{\bf Construction  of zero modes}             
\label{sec:Conjecture}          

We are going to prove that (\ref{eqn:induction1}) is true for 
every $m \in \mathbb N$, and describe how to construct
the sequence $\{ \psi^{(m)}\}_{m=1}^{\infty}$ of zero modes
in terms of a root of the polynomial equation $P_m(b_0^2)=0$.

Recall that in Proposition \ref{prop:Prop2.2}  we saw that
$a_j=p_j(b_0^2)$ and $b_k=b_0 q_k(b_0^2)$ with the polynomials
$p_j(t)$ and $q_k(t)$  of degrees $j$ and $k$ respectively, where
$1 \le j, \, k \le m$. 
Let 
\begin{equation}
\begin{cases}
c_j:= p_j(0)  \\
d_k:= \ds{\lim_{t \to \infty} \frac{q_k(t)}{t^k}}. 
\end{cases}
\end{equation}
In other words,
$c_j$ denotes the constant coefficient of the polynomial $p_j(t)$
and $d_k$ denotes the coefficient of $t^k$ in the polynomial $q_k(t)$,
of which degree is $k$.

\begin{lem}   \label{lem:Lem4.1}
 Let $m$ be a fixed nonnegative integer and 
let $a_n$ and $b_n$ be the coefficients in
     {\rm (\ref{1psim})} with $a_0=1$.                                                           
        Then we have
\begin{equation}\label{4cmcoeff}                                                 
         c_m = \frac{5 \cdot 7 \cdot 9 \cdots (2m + 3)}{2^m(m!)}
\end{equation}
     and
\begin{equation}\label{4dmcoeff}                                                 
        d_m = (-1)^m\frac{3^{2m}}{5 \cdot 7 \cdot 9 \cdots (2m + 3)2^m(m!)}.
\end{equation}
\end{lem}                                                              

\begin{proof}     
We divide the proof into two steps.

        (I) From the equation $(2j - 1)$ of ($L_m$) we have
\begin{equation*}
        2ja_j - (2m + 5 - 2j)a_{j-1} = -3b_0b_{j-1},
\end{equation*}
     where the right hand side has no constant coefficient 
      as a polynomial of $b_0$.
     Hence we have
\begin{equation*}
         2jc_j - (2m + 5 - 2j)c_{j-1} = 0.
\end{equation*}
Therefore we obtain recursive relations
\begin{equation}\label{4cj0rec}                                          
            c_j = \frac{2m + 5 - 2j}{2j}\,c_{j-1} \q (j =2, \cdots, m),
\end{equation}
which implies that
\begin{equation}\label{4cj0rec1}                                         
    c_j 
  = \frac{\big(2m+5-2j\big)\big(2m+5-2(j-1)\big)\cdots
  \big(2m+1\big)\big(2m+3\big)}{2^j(j!)},
\end{equation}
where we should note that $c_1 = (2m+3)/2$. 
We obtain (\ref{4cmcoeff}) by  setting $j = m$ in
     (\ref{4cj0rec1}).

\SP

 (II) Let $\wc_j$ be the coefficient of $t^j$ of $p_j(t)$. Then it follows  
 from the equation $(2j - 1)$
     of ($L_m$) that
\begin{equation*}
           2j\wc_j = -3d_{j-1}
\end{equation*}
 or
\begin{equation}\label{4djrec}                                            
           \wc_j = -\frac3{2j}\,d_{j-1}.
\end{equation}
On the other hand, from the equation $(2k)$ of ($L_m$)
 with $k$ replaced by $j$ we see that
\begin{equation*}
          (2j + 3)d_j = 3\wc_j.
\end{equation*}
 Thus we have
\begin{equation}\label{4djrec1}                                                      
            d_j = \frac{3}{2j + 3}\,\wc_j.
\end{equation}
It follows from (\ref{4djrec}) and (\ref{4djrec1}) that
\begin{equation*}
  d_j = -\frac{3}{2j + 3}\frac3{2j}\,d_{j-1} = -\frac{3^2}{2j(2j + 3)}\,d_{j-1}
\end{equation*}
     for $j = 2, \cdots, m$, which implies that
\begin{equation}\label{4djrec2}                                                        
    d_j = (-1)^j\frac{3^{2j}}{5 \cdot 7 \cdot 9 \cdots (2j + 3)2^j(j!)}\,,
\end{equation}
  where we should note that $d_1=(-9)/(5\cdot 2)$. 
 We  thus obtain (\ref{4dmcoeff}) by  setting $j = m$ in (\ref{4djrec2}).
\end{proof}

\begin{thm}  For each $m \in \mathbb N$, we have
\begin{equation}  \label{eqn:roots-1ume} 
\cR_m= 
\Big\{ 1, \, \Big( \frac{5}{3} \Big)^2, \, \Big( \frac{7}{3} \Big)^2, \,
\cdots, \, \Big( \frac{2m+3}{3} \Big)^2 \Big\}.
\end{equation}
\end{thm}                                                                 

\begin{proof}   
We prove the theorem by induction on $m$.
As was pointed out in Remark \ref{rem:Rem3.3},
we only have to show that
\begin{equation}   \label{eqn:induction4}
\Big( \frac{2m+3}{3}  \Big)^2 \in \cR_m
\end{equation}
for every $m \ge 2$,
assuming that (\ref{eqn:roots-1ume}) with $m$ replaced by
$m-1$ is true.

Let us recall that
\begin{equation}  \label{eqn:induction5}
         b_0b_m - a_m =b_0^2 \, q_m(b_0^2) - p_m(b_0^2) = P_m(b_0^2),
\end{equation}
where $p_m(t)$ and $q_m(t)$ are polynomials of degree $m$.
It follows from (\ref{eqn:induction5}) that $d_m$ is equal
to the coefficient of $t^{m+1}$ of the polynomial $P_m(t)$,
of which degree is $m+1$. Also it 
 follows from (\ref{eqn:induction5}) that the constant
 coefficient of  $P_m(t)$ is given by  $-c_m$.

By hypothesis of the induction, we have
\begin{equation}  \label{eqn:roots-2ume} 
\cR_{m-1}= 
\Big\{ 1, \, \Big( \frac{5}{3} \Big)^2, \, \Big( \frac{7}{3} \Big)^2, \,
\cdots, \, \Big( \frac{2m+1}{3} \Big)^2 \Big\}.
\end{equation}
Since $\cR_{m-1} \subset \cR_m$ by Proposition \ref{prop:Prop3.1},
we see that
\begin{equation*}
1, \, \Big( \frac{5}{3} \Big)^2, \, \Big( \frac{7}{3} \Big)^2, \,
\cdots, \, \Big( \frac{2m+1}{3}  \Big)^2
\end{equation*}
are  the roots of $P_m(t)$.
For simplicity, we put
\begin{equation}\label{4lambda}                                                     
        \lambda_j = \Big(\frac{2j + 1}3\Big)^2 \q (j = 1, 2, \cdots, m).
\end{equation}
Since there exists  one more root $\lambda \in \cR_m$  of
$P_m(t)$, 
we find that
\begin{equation}\label{4compare}                                                           
    P_m(t) = d_m(t - \lambda_1)(t - \lambda_2) \cdots (t - \lambda_m)(t - \lambda).
\end{equation}
Noting that $P_m(0)= -c_m$,
 we get
\begin{equation*}
           d_m(-1)^{m+1}\lambda_1\lambda_2 \cdots \lambda_m\lambda = -c_m.
\end{equation*}
Hence, by using (\ref{4cmcoeff}), (\ref{4dmcoeff}) and 
(\ref{4lambda}), we obtain
\begin{equation*}
   (-1)^{m+1}\Big[\prod_{j=1}^m \Big(\frac{2j + 1}3\Big)^2\Big]\,\lambda
                        = -\frac{c_m}{d_m}
                          = (-1)^{m+1}\prod_{j=1}^{m+1} \Big(\frac{2j + 1}3\Big)^2.
\end{equation*}
Therefore we can conclude that
\begin{equation*}
       \lambda = \Big(\frac{2m + 3}3\Big)^2,
\end{equation*}
     which implies  (\ref{eqn:induction4}).
\end{proof}

For each $m \in \mathbb N$, the polynomial (\ref{eqn:induction5}) has
$2m + 2$ roots:
\begin{equation*}
b_0= \pm 1, \, 
\pm\frac{5}{3},  \, \cdots, \, \pm\frac{2m+3}{3}.
\end{equation*}
If we choose the root $b_0=+(2j+1)/3$ for a fixed $j$ 
with $1 \le j \le m+1$, then 
we
can define $a_1$, $\cdots$, $a_m$, $b_1$, $\cdots$, $b_m$
by Proposition \ref{prop:prop2.4}.
With these $a_1$, $\cdots$, $a_m$, $b_1$, $\cdots$, $b_m$ obtained,
we construct 
$\psi_{j,+}^{(m)}(x)$ by (\ref{1psim}).
It follows from Propositions \ref{prop:Prop2.1} and \ref{prop:LY}
that $\psi_{j,+}^{(m)}$ is a zero mode
of the magnetic Dirac operator 
$H_{A^{(m)}_{j,+}}:=\sigma \cdot (D - A_{j,+}^{(m)})$,
where 
$A_{j,+}^{(m)}$ is defined by (\ref{eqn:LossYau1})
with
\begin{equation*}
h(x)=\frac{2j+1}{\langle x \rangle^2}, \quad
\psi(x) = \psi_{j,+}^{(m)}(x).
\end{equation*}

The sequence $\{ \psi^{(m)} \}_{m=1}^{\infty}$ constructed 
in Adam, Muratori and Nash \cite{AdamMuratoriNash1} 
is now obtained by putting
\begin{equation*}
\psi^{(m)}(x) := \psi_{m+1,+}^{(m)}(x).
\end{equation*}

We make a few of concluding remarks.

\begin{itemize}
\item[(i)]
For each $m \in \mathbb N$,  set
\begin{equation}\label{4psim}                                                         
   \Psi_m = \big\{ \psi_{j,+}^{(m)}(x) \; \big| \;
    j = 1, 2, \cdots, m + 1 \} .
\end{equation}
Adam, Muratori and Nash \cite{AdamMuratoriNash1} pointed out 
that 
\begin{equation*}
\Psi_1 \subsetneqq \Psi_2 \subsetneqq  
\cdots \subsetneqq \Psi_m  \subsetneqq \cdots.
\end{equation*}
\item[(ii)]
In a similar manner, 
choosing  the root $b_0=-(2j+1)/3$,  we can construct 
a different sequence of zero modes from 
$\{ \psi^{(m)} \}_{m=1}^{\infty}$ defined above. 
\end{itemize}

\vspace{15pt}
%%%%%%%%%%%%%%%%%%%%%%%%%%%%%%%%%%%%%%%%%%%%%%%%%%%%%%%%%
%%%%%%%%%%%%%%%%%%%%%%%%%%%%%%%%%%%%%%%%%%%%%%%%%%%%%%%%%
%%%%%%%%%%%%%%%%%%%%%%%%%%%%%%%%%%%%%%%%%%%%%%%%%%%%%%%%%

\textbf{Acknowledgments.} T.U. would like to thank
Ryuichi Ashino for his help with our Maple program,
and thank 
Takeshi Usa for his valuable comment  
on the solvability of the Adam-Muratori-Nash polynomials
that Galois theory seems to play an important role
 behind the scene.

\vspace{10pt}
%%%%%%%%%%%%%%%%%%%%%%%%%%%%%%%%%%%%%%%%%%%%%%%%%%%%%%%%%
%%%%%%%%%%%%%%%%%%%%%%%%%%%%%%%%%%%%%%%%%%%%%%%%%%%%%%%%%
%%%%%%%%%%%%%%%%%%%%%%%%%%%%%%%%%%%%%%%%%%%%%%%%%%%%%%%%%

%$

\end{document}